\documentclass{article}

\usepackage{amsmath,amsfonts,amsthm,amssymb,amscd,color,xcolor,mathrsfs,verbatim,microtype}
\usepackage{graphicx,eurosym}
\usepackage{hyperref}
\usepackage{mathtools}

\usepackage[applemac]{inputenc}

\usepackage[cyr]{aeguill}

\colorlet{darkblue}{blue!50!black}

\hypersetup{
    colorlinks,%
    citecolor=blue,%
    filecolor=red,%
    linkcolor=darkblue,%
    urlcolor=blue,%
    pdfnewwindow=true,%
    pdfstartview={FitH}
}

\usepackage{graphicx,amscd,mathrsfs,wrapfig,mathrsfs,lipsum}
\usepackage{eufrak}
\usepackage{float}
\usepackage{tikz}
\usepackage{multicol}
\usepackage{caption}
\usetikzlibrary{arrows}
\usepackage{capt-of}

\colorlet{darkblue}{blue!50!black}

\newcommand{\p}{\partial}
\newcommand{\e}{\varepsilon}

\newcommand{\C}{{\mathbb C}}

\DeclarePairedDelimiter\ceil{\lceil}{\rceil}

\newcommand{\mek}{{\mathbf 1}}

\newcommand{\R}{{\mathbb R}}
\newcommand{\Z}{{\mathbb Z}}

\newcommand{\pP}{{\mathbb P}}
\newcommand{\I}{{\mathbb I}}
\newcommand{\E}{{\mathbb E}}
\newcommand{\T}{{\mathbb T}}

\newcommand{\N}{{\mathbb N}}
\newcommand{\la}{\lambda}

\newcommand{\ty}{\infty}

\newcommand{\de}{\delta}

\newcommand{\EE}{{\cal E}}
\newcommand{\FF}{{\cal F}}

\newcommand{\HH}{{\cal H}}
\newcommand{\II}{{\cal I}}

\newcommand{\KK}{{\cal K}}

\newcommand{\RR}{{\cal R}}

\newcommand{\sS}{{\cal S}}
\newcommand{\TT}{{\cal T}}

\newcommand{\lag}{\langle}
\newcommand{\rag}{\rangle}

\newcommand{\dd}{{\textup d}}

\newcommand{\B}{{\mathbb B}}

\newcommand{\DB}{{\mathbb D}}

\newcommand{\lspan}{\mathop{\rm span}\nolimits}

\newcommand{\supp}{\mathop{\rm supp}\nolimits}

\theoremstyle{plain}

\newtheorem*{mta}{Theorem~\hypertarget{A}{A}}
\newtheorem*{mtb}{Theorem~\hypertarget{B}{B}}
\newtheorem*{mtc}{Theorem~\hypertarget{C}{C}}

\newtheorem*{lemma*}{Lemma}
\newtheorem{theorem}{Theorem}[section]

\newtheorem{proposition}[theorem]{Proposition}

\theoremstyle{definition}
\newtheorem{definition}[theorem]{Definition}

\theoremstyle{remark}

\numberwithin{equation}{section}

\begin{document}

\author{Alessandro Duca\,\footnote{Universit\'e Paris-Saclay, UVSQ, CNRS, Laboratoire de Math\'ematiques de Versailles, 78000, Versailles, France;  e-mail: \href{mailto:alessandro.duca@uvsq.fr}{Alessandro.Duca@uvsq.fr}} \and  Vahagn~Nersesyan\,\footnote{Universit\'e Paris-Saclay, UVSQ, CNRS, Laboratoire de Math\'ematiques de Versailles, 78000, Versailles, France;  e-mail: \href{mailto:vahagn.nersesyan@uvsq.fr}{Vahagn.Nersesyan@uvsq.fr}}}

 \date{\today}

\title{Bilinear  control  and growth of Sobolev norms     for   the  nonlinear  Schr\"odinger~equation}
\date{\today}
\maketitle
 
 \bigskip

 \bigskip

\begin{abstract}

We consider  the nonlinear Schr\"odinger equation (NLS) on a torus of arbitrary   dimension.~The equation is 
studied    in presence of an external      potential field  whose    time-dependent amplitude is taken as control. 
Assuming   that the   potential satisfies a   saturation property, we show that the NLS equation is approximately 
controllable between any pair of eigenstates in arbitrarily small  time.~The~proof is obtained by  developing a 
multiplicative version of  a geometric control approach  introduced by  Agrachev~and~Sarychev. We~give an application of this   result to the study of the large time   behavior of  the NLS equation  with random 
potential.~More~precisely, we~assume that the amplitude of the potential is a   
random process whose law is   $1$-periodic in time and non-degenerate.~Combining the controllability with  a    stopping time argument and the Markov property, we show that     the trajectories  of the random    equation are  almost surely unbounded in   regular Sobolev~spaces.

\medskip
\noindent
{\bf AMS subject classifications:}      35Q55,     35R60, 37L55, 81Q93, 93B05

\medskip
\noindent
{\bf Keywords:} Nonlinear Schr\"odinger equation, approximate   controllability, geometric control theory,     growth of Sobolev norms, random perturbation

\end{abstract}

\newpage
  \tableofcontents

\setcounter{section}{-1}

\section{Introduction}
\label{S:0}

 In this paper, we study the  controllability and the   growth of Sobolev norms for  the following  nonlinear 
Schr\"odinger (NLS) equation  on the     torus~$\T^d=\R^d/2\pi\Z^d$:
      \begin{equation}\label{0.1}
 	i \p_t \psi=-\Delta \psi +V(x)\psi  +\kappa  |\psi|^{2p} \psi +\lag u(t),Q(x) \rag \psi.   
 \end{equation}We assume that  $V:\T^d\to \R$ is an arbitrary smooth potential,   $Q:\T^d\to \R^q$ is a  given 
smooth   external        field
   subject to some geometric condition, $d,p\ge1$ are   arbitrary  
integers, and $\kappa $ is~an arbitrary   real number.~The role of~the~control (or~the~random perturbation) is  
played  by   $\R^q$-valued function (or random process) $u$ which is assumed to depend only on 
time.~Eq.~\eqref{0.1} is equipped with the initial condition 
  \begin{equation}\label{0.2}
  \psi(0,x)=\psi_0(x)	
  \end{equation}belonging to a Sobolev space $H^s=H^s(\T^d;\C)$ of order
  $s> d/2$, so that the problem   is  locally      well-posed. 
  
  The purpose   of this paper is to study the NLS equation~\eqref{0.1} when the driving   force $u$ acts  multiplicatively through only  few    low Fourier modes.~Referring the reader to the  subsequent sections  for the general setting, let us formulate in this Introduction   particular cases of our main  results.~Let $\KK\subset \Z^d_*$ be the set of $d$ vectors  defined~by
   \begin{equation}\label{0.3}
  	 \KK= \{(1,0,\ldots, 0),\, (0,1,\ldots, 0),\, \ldots,\,(0,0,\ldots, 1, 0),\, (1,\ldots, 1)\},
  	 \end{equation} 
  	 and assume that the    field~$Q=(Q_1, \ldots, Q_q)$ is such that    
\begin{equation}\label{0.4}
  \left\{ {\bf 1},\,
  \sin\lag x,k\rag,\, \cos\lag x,k\rag: k\in \KK\right\}\subset \lspan \{Q_j:j=1,\ldots, q\}.   	
\end{equation} Let $s_d$ be the   least integer strictly greater than $d/2$.
  \begin{mta}
  	The problem~\eqref{0.1},~\eqref{0.2} is approximately controllable in the following sense:    for any $s\ge s_d$,   $\e>0$,   $\varkappa>0$,    $\psi_0 \in   H^s$, and   $ \theta\in   C^\ty(\T^d;\R)$,   there is a time $T \in (0,\varkappa)$,   a control $u\in L^2([0,T]; \R^q)$,  and a unique       solution $\psi\in C([0,T]; H^s)$ of~\eqref{0.1},~\eqref{0.2} such that   
	$$ 
 \|\psi(T) - e^{i\theta}\psi_0\|_{H^s}<\e.
 $$  
   	  \end{mta} 
   More general formulation of this result is given in Theorem~\ref{T:2.2}, where the controllability is proved under       an abstract        saturation   condition    for the field~$Q$      (see~Condition~(\hyperlink{H1}{\rm H$_1$})). Note~that the~time~$T$   may  depend  on the initial condition~$\psi_0$, the target~$e^{i\theta}\psi_0$, and  the 
     parameters in the equation.~In the second result, we show that, when~$V= 0$ and~$\psi_0$ is an 
eigenstate~$\phi_l(x)=(2\pi)^{-d/2}e^{i\lag x,l\rag},$~$l\in \Z^d$ of~the Laplacian, the system can be 
approximately controlled in   any fixed time~$T>0$
     to any target of the form~$e^{i\theta}\phi_m$ with    $m\in \Z^d$.
  	    \begin{mtb}
    For any $s\ge s_d,$ $\e>0$,   $l, m\in \Z^d$,~$ \theta\in   C^\ty(\T^d;\R)$,     and~$T>0$, there is a   control $u\in L^2([0,T]; \R^q)$ and a    unique  solution $\psi\in C([0,T]; H^s)$ of~\eqref{0.1},~\eqref{0.2} with  $V= 0$ and $\psi_0=\phi_l$ such that  
	$$ 
 \|\psi(T) - e^{i\theta}\phi_m\|_{L^2}<\e.
 $$   \end{mtb}

 The controllability of the Schr\"odinger  equation with time-dependent bilinear (multiplicative) control  has 
attracted a lot of attention during the last fifteen years.~In~the one-dimensional case, local exact 
controllability results are   established by Beauchard, Coron, and Laurent~\cite{B-2005, BC-2006, 
BL-2010}.~There~is a vast literature on  the approximate controllability  in the multidimensional case. 
For~the~first achievements, we refer the reader to the papers by Boscain et~al.~\cite{CMSB-2009,  BCCS-12}, 
Mirrahimi~\cite{M-09},        and the second author~\cite{VN-2010}. Except the paper~\cite{BL-2010}, all the other  
 works    consider the linear Schr\"odinger equation, i.e., the one obtained  
by taking~$\kappa=0$ in Eq.~\eqref{0.1}; note~that   in that case  the control problem is still nonlinear in $u$.

   Theorems~A and~B are the first to deal with  the problem of bilinear approximate controllability of the NLS equation.~Let 
us emphasise that the   controllability between any pair of eigenstates in arbitrarily small  time is  new    even 
in the linear case~$\kappa=0$.~It~is interesting to note that Theorem~B complements a   result by Beauchard et 
al.~\cite{BCT2}, which proves that, for some choices of the field $Q$, there is a minimal time for the approximate controllability to some particular states in the phase space.

  The approach adopted in   the proofs of    Theorems~A and~B is quite different from those   used in the 
literature for bilinear control systems. We proceed by developing  Agrachev--Sarychev type arguments which were 
previously  employed  in the case of additive     controls. Let us recall that Agrachev and 
Sarychev~\cite{AS-2005, AS-2006} considered the global approximate controllability of the 2D~Navier--Stokes and 
Euler systems.~Their approach  has been further extended by many authors to different equations.    Let us mention, 
for example, the papers~\cite{shirikyan-cmp2006, shirikyan-aihp2007} by Shirikyan who   considered the approximate 
controllability of
    the 3D~Navier--Stokes system    and Sarychev~\cite{Sar-2012} who considered   the  case of the  2D defocusing 
cubic Schr\"odinger equation.~The~configuration we use in the present paper is closer to the one elaborated in the 
recent paper~\cite{VN-2019A}, where parabolic PDEs are studied with polynomial nonlinearities.~We refer the reader 
to   the reviews~\cite{AS-2008, Shir-2018} and the paper~\cite{VN-2019A} for more references and discussions.

The present paper is the first to deal with Agrachev--Sarychev type arguments in a bilinear setting.  
  To explain the scheme   of the proof  of Theorem~A, let~us denote by $\RR_t(\psi_0, u)$ the solution of problem~\eqref{0.1},~\eqref{0.2}   defined up to some maximal time. A central role in the proof is played by the limit
  \begin{equation}\label{0.5}
 	e^{-i\delta^{-1/2}\varphi} \RR_{\de}(e^{i\delta^{-1/2}\varphi}\psi_0,\de^{-1} u)\to  e^{-i \left(\B(\varphi)+\lag u,Q \rag\right)}\psi_0  \quad\text{in $H^s$ as $\de\to 0^+$}
\end{equation}which holds for any $\psi_0\in H^s$, $\varphi\in   C^\ty(\T^d;\R)$, and constant   $u\in \R^q$. Here  we denote	$\B(\varphi)(x)=\sum_{j=1}^d\left(\p_{x_j} \varphi(x)\right)^2$.~Applying this limit with $\varphi=0$ and using the assumption~\eqref{0.4},  we see that the equation  can be controlled in small time from initial point~$\psi_0$ arbitrarily close to       $e^{i\theta}\psi_0$ for  any~$\theta$ in     the vector~space
$$
\HH_0  =\lspan\left\{ {\bf 1},\,
  \sin\lag x,k\rag,\, \cos\lag x,k\rag: k\in \KK\right\}.
$$ 
  By applying again the limit~\eqref{0.5} with   functions $\varphi=\theta_j \in\HH_0 $,   $j=1,\ldots, n$, 
we add more directions in $\theta$.  That is, we show that the system can be steered   from~$\psi_0$   close to~$e^{i\theta}\psi_0$, where~$\theta$  now belongs to a larger vector  space~$\HH_1$  whose elements are of the~form 
$$
 \theta_0-\sum_{j=1}^n \B(\theta_j).
$$
We iterate this argument and   construct an increasing   sequence of subspaces~$\{\HH_j\}$ 
such that the equation can be approximately controlled to any target~$e^{i\theta}\psi_0$
 with any $\theta \in  \HH_j$ and $j \ge1$.~Using    trigonometric computations, we show that  
   the union~$\cup_{j=1} \HH_j $ is dense in~$C^k(\T^d,\R)$ for any $k\ge1$ (in~other words,~$\HH_0$ is a   saturating space for the NLS equation, see Definition~\ref{D:2.1}). This~completes the proof of~Theorem~A. 
   
   Theorem~B is derived from Theorem~A by noticing that the eigenstate~$\phi_l$  can be approximated in~$L^2$ 
by functions of the form~$e^{i\theta}\phi_m$ and that the eigenstates are  constant solutions\footnote{This follows immediately  from the assumptions  that~$V= 0$ and~$\mek \in 
\HH_0$.} of 
Eq.~\eqref{0.1} corresponding to some  control. This   allows to    appropriately adjust the controllability  time  and choose  it the same for any initial 
 condition and   target.

 As an application of Theorem~A, we study the large time behavior of the trajectories of  the  random NLS equation.~We show that if a random process perturbes   the same    Fourier modes as in the above controllability results, 
  then  the   energy is almost surely    transferred   to higher modes resulting in the unboundedness of the 
trajectories in regular  Sobolev spaces.~More precisely, we replace the control~$u$   by an~$\R^q$-valued random 
process~$\eta$   of the~form
\begin{equation}\label{0.6}
\eta(t)=\sum_{k=1}^{+\infty} \I_{[k-1,k)}(t) \eta_k(t-k+1),
\end{equation}where $\I_{[k-1,k)}$ is the indicator function of the interval $[k-1,k)$ and 
  $\{\eta_k\}$ are   independent identically distributed   random variables in~$L^2([0,1];\R^q)$ with non-degenerate law (see Condition~(\hyperlink{H2}{\rm H$_2$})).~The solution $\psi$ of the problem~\eqref{0.1},~\eqref{0.2}, \eqref{0.6} will be itself a random process in~$H^s$.
   We prove the following result.

   \begin{mtc}
  	For any $s>s_d$ and any non-zero~$\psi_0\in H^s$, the trajectory of~\eqref{0.1}, \eqref{0.2}, \eqref{0.6} is almost surely unbounded in~$H^s$.
  	 \end{mtc}  
  The idea of constructing unbounded solutions by using random perturbations is not new.~Such results have been obtained by Bourgain~\cite{Bourgain-1999} and  Erdogan et al.~\cite{EBK-2003} for linear one-dimensional Schr\"odinger equations.~They also provided polynomial lower bounds   for        the growth.~Unboundedness of trajectories for   multidimensional linear Schr\"odinger equations  
  	  is obtained in~\cite{VN-2009}.~In~that paper, the assumptions on the law of the random perturbation are rather general and   no estimates for the growth are given; Theorem~C is a generalisation of that result to the  case of  the NLS equations.~There are also   examples of linear Schr\"odinger equations with various deterministic time-dependent potentials which admit unbounded trajectories:~e.g., see the papers by Bambusi et al.~\cite{MR3738262}, Delort~\cite{Delort-14}, Haus and   Maspero~\cite{HM-20,Maspero-19}, and the references therein.

  	    	  There are only few results   in the case of   unperturbed NLS equations. For~cubic defocusing Schr\"odinger equations on bounded domains or manifolds, the existence of unbounded trajectories in regular Sobolev spaces is a challenging open problem (see Bourgain~\cite{Bourgain-2000}).~In different situations,   existence of trajectories with arbitrarily large finite
  	    	   growth has been shown by Kuksin~\cite{SBK-97}, 
 Colliander et al.~\cite{CKSTT-2010}, Guardia and Kaloshin~\cite{GK-2015}, and others.~Hani et al.~\cite{HPT-2015}    show the existence of unbounded trajectories in the case of the cubic defocusing Schr\"odinger equation   on the infinite cylinder~$\R\times\T^d$. In the case of the cubic Szeg\H{o} equation on the circle,  G\'{e}rard and Grellier~\cite{GG-2017} show that the trajectories are generically unbounded in Sobolev spaces. Moreover, they exhibit   the existence of a family  of solutions with superpolynomial growth.

 Let us give a brief (and not entirely  accurate) description of the main ideas of the proof of 
    Theorem~C. By starting  from any initial point~$\psi_0\in H^s$,      Theorem~A allows to increase  the Sobolev 
norms      by choosing   appropriately the control.  This,~together with a compactness argument and the assumption 
that the law of the process~$\eta$ is non-degenerate,
 leads to    a uniform estimate of the form
  $$
 c_M= \sup_{\psi_0\in H^s}\pP\left\{\sup_{t\in[0,1]}\|\psi(t)\|_{H^s}>M\right\}<1 
  $$ for any $M>0$. By combining the latter with the Markov property, we show that
  $$
   \pP\left\{\sup_{t\in[0,n]}\|\psi(t)\|_{H^s}>M\right\}\le c_M^n
  $$for any $\psi_0\in H^s$.
   Then, the Borel--Cantelli lemma implies that the norm of any trajectory becomes almost surely larger than  $M$   
 in some random   time that is almost  surely finite. As~$M$ is arbitrary, this proves the required result.

 The paper is organised as follows.~In Section~\ref{S:1}, we discuss  the    local  well-posedness and some stability properties  of  the NLS equation.~In Section~\ref{S:2}, we~formulate more general versions of Theorems~A and~B and give their proofs.  Section~\ref{S:3} is devoted to the derivation of limit~\eqref{0.5}.~In Section~\ref{S:4}, we establish a general criterion for the  validity of the    saturation property.~Finally, in Section~\ref{S:5}, we~prove~Theorem~C.

 \subsubsection*{Acknowledgement}  
 
 The authors thank Armen Shirikyan for his valuable comments.
 The research of AD  was supported by the ANR     grant ISDEEC
ANR-16-CE40-0013.~The research of VN was supported by the ANR   grant NONSTOPS
ANR-17-CE40-0006-02.

 \subsubsection*{Notation}

  In what follows, we use the following notation.
 
\smallskip
\noindent
 $\lag \cdot,  \cdot\rag$ is the Euclidian   scalar product in $\R^q$ and $\|\cdot\|$ is the corresponding norm. We   write $m \bot l$ when the vectors~$m, l\in \R^q$ are orthogonal and $m \not\perp l$ when they are not.

\smallskip
\noindent
$H^s=H^s(\T^d;\C),~s\ge 0$ and $L^p=L^p(\T^d;\C),~p\ge 1$ are the standard     Sobolev and Lebesgue  spaces of functions $f:\T^d\to \C$  endowed with the    norms~$\|\cdot\|_s$ and~$\|\cdot\|_{L^p}$.  The space $L^2$ is endowed with the scalar product 
$$
\lag f,g \rag_{L^2}=\int_{\T^d} f(x)\overline {g(x)} \dd x.
$$

  \noindent
$C^s=C^s(\T^d;\C)$, $s\in \N\cup \{\ty\}$ is the space of $s$-times continuously differentiable functions $f:\T^d\to \C$.

\smallskip
\noindent
Let  $X$ be a Banach  space. We denote by $B_X(a,r)$    the closed ball of radius $r>0$ centred at $a\in X$.

 \smallskip
\noindent
We   write $J_T$ instead of $[0,T]$ and $J$ instead of $[0,1]$.

\smallskip
\noindent
$C(J_T;X)$  is the space of continuous functions $f:J_T\to X$   with the norm 
$$
\|f\|_{C(J_T;X)}=\max_{ t\in  J_T} \|f(t)\|_X.
$$ 

 \noindent
$L^p(J_T;X), 1\leq p<\infty$   is      the
space of Borel-measurable functions $f: J_T \to  X$ with  
$$
\|f\|_{L^p(J_T;X)}=\left(\int_{0}^T \|f(t)\|_X^p\dd t
\right)^{1/p}<\ty.
$$

\noindent
  $\ceil x $ is the least integer greater than or equal to  $x \in \R$.

\smallskip
\noindent
$s_d$ is the   least integer strictly greater than $d/2$.

\smallskip
\noindent
    ${\bf 1}$ is the  function identically equal~to~$1$ on $\T^d$.
 
  \section{Preliminaries}\label{S:1}

 In this section, we   consider  the    NLS  equation~\eqref{0.1},
 where~$u$ is a   deterministic $\R^q$-valued  function and  $V:\T^d\to \R$ and $Q:\T^d\to \R^q$ are  arbitrary smooth functions. In~what follows, we shall always assume that the parameters~$d\ge1$,~$p\ge1$, and~$\kappa\in \R $   are~arbitrary.~Here we formulate   two propositions that will be used in the proofs of our main results.~The first one
  gathers some well-known facts about the local well-posedness and stability of the NLS equation in regular Sobolev~spaces.
       \begin{proposition}\label{P:1.1}
For any~$s>d/2$,    $\hat \psi_0\in H^s$, and   $\hat u\in L^2_{loc}
(\R_+; \R^q)$,   there is a maximal time     $\TT=\TT(\hat \psi_0, \hat u)>0$ and a unique solution $\hat \psi$ of the problem \eqref{0.1},~\eqref{0.2} with~$(\psi_0, u)=(\hat\psi_0, \hat u)$     whose restriction to  the interval $J_T$  belongs  to~$C(J_T; H^s)$ for any $T<\TT$.~If~$\TT<\ty$, then~$\|\hat \psi(t)\|_s\to +\infty$  as $t\to \TT^-$. Furthermore, for any~$T<\TT$,  there are   
  constants
 $\delta=\delta(T,\Lambda)>0$  and $C=C(T,\Lambda)>0$, where   
$$
\Lambda= \|\hat \psi\|_{C(J_T;H^s)}+\|\hat u\|_{L^2(J_T;\R^q)},
$$
such that the following two properties hold.

\begin{enumerate}
\item[(i)] For any $\psi_0\in H^s $   and $u\in L^2(J_T; \R^q)$  satisfying  
\begin{equation}\label{1.1}
\|\psi_0-\hat \psi_0\|_s+ \|u-\hat u\|_{L^2(J_T;\R^q)}  <
\delta,
\end{equation}
  the problem \eqref{0.1},~\eqref{0.2} has a unique solution
$\psi\in C(J_T;H^s).$
\item[(ii)] 
  Let $\RR$ be the resolving operator for Eq.~\eqref{0.1}, i.e., the mapping  taking a couple $(\psi_0,u)$ satisfying \eqref{1.1} to the solution~$\psi$.~Then     
 \begin{align*}
\|\RR(\psi_0, u) -\RR(\hat \psi_0,\hat u)\|_{C(J_T;H^s)}&\le 
C \left( \|\psi_0-\hat \psi_0\|_s+
\|u-\hat u\|_{L^2(J_T; \R^q)} \right).
 \end{align*}
 \end{enumerate}
\end{proposition}
The proof of this proposition is rather     standard, so we omit it   (e.g., see  Section~3.3 in~\cite{TT-2006} or Section~4.10 in~\cite{C-2003} for similar results).~Let $\sS$ be the  unit sphere in  $L^2$. As the functions~$V, Q$, and~$u$ are real-valued, the solution~$\psi$   belongs   to~$\sS$ throughout its lifespan, provided that~$\psi_0\in \sS \cap H^s$.

 Before  formulating  the second proposition, let us introduce some notation. 
For~any~$\psi_0\in  H^s$ and $T>0$, let~$\Theta(\psi_0,T)$ be the set of functions~$u\in
L^2(J_T;\R^q)$ such that the problem~\eqref{0.1},~\eqref{0.2} has
a solution in~$C(J_T;H^s)$. By the previous   proposition,   the set   $ \Theta(\psi_0,T)$ is      open     in $L^2(J_T;\R^q)$.~For any   $\varphi \in C^1(\T^d;\R)$, let
 \begin{equation}\label{1.2}
 	\B(\varphi)(x)=\sum_{j=1}^d\left(\p_{x_j} \varphi(x)\right)^2. 
 \end{equation}
We have the following asymptotic property in small time.
   \begin{proposition}\label{P:1.2}   
 For any $s\ge s_d$,  $\psi_0  \in H^s$,       $u \in \R^q$, and $\varphi \in C^r(\T^d;\R)$, where~$r={\ceil s+2}$,  there is a constant~$\de_0>0$ such that\,\footnote{For any vector $u\in \R^q$, with a slight abuse  of notation, we denote by the same letter the constant function equal to $u$.} $ \de^{-1} u \in  \Theta(e^{i\delta^{-1/2}\varphi}\psi_0, \de)$ for any $\de\in (0,\de_0)$ and     the following  limit  holds  
\begin{equation}\label{1.3}
 	e^{-i\delta^{-1/2}\varphi} \RR_{\de}(e^{i\delta^{-1/2}\varphi}\psi_0,\de^{-1} u)\to  e^{-i \left(\B(\varphi)+\lag u,Q \rag\right)}\psi_0  \quad\text{in $H^s$ as $\de\to 0^+$},
\end{equation}where $\RR_{\de}$ is the restriction of the solution at time $t=\delta$.
\end{proposition} The proof of this proposition   is  postponed to    Section~\ref{S:4}.  Limit~\eqref{1.3} is a multiplicative version of a limit established in Proposition~2 in~\cite{VN-2019A} in the case of parabolic PDEs with additive controls.

  \section{Approximate controllability}\label{S:2}

 In what follows, we   assume that  $s\ge s_d$ and denote~$r={\ceil s+2}$ as in Proposition~\ref{P:1.2}.~We~start this section with  a 
definition of a saturation property     inspired by the papers  \cite{AS-2006, shirikyan-cmp2006}. Let~$\HH$ be a 
finite-dimensional subspace of~$ C^r(\T^d;\R)$, and let~$\FF(\HH)  $ be the largest   subspace of $ C^r(\T^d;\R)$  
whose  elements    can be represented in the~form    
$$
 \theta_0-\sum_{j=1}^n \B(\theta_j)
$$for some  integer   $n\ge1$ and  functions  $ \theta_j\in \HH$, $j=0, \ldots, n$, where $\B$ is given 
by~\eqref{1.2}.~As~$\B$ is   quadratic,~$\FF(\HH)$ is well-defined and finite-dimensional.
 Let~us~define a~non-decreasing sequence~$\{\HH_j \}$   of     finite-dimensional subspaces     by~$\HH_0= \HH$  and    $\HH_j=\FF(\HH_{j-1})$,      $j \geq1$, and denote 
\begin{equation}\label{2.1}
  \HH_\ty=  \bigcup_{j=1}^{+\infty} \HH_j.
\end{equation}
  \begin{definition}\label{D:2.1} 
  A finite-dimensional subspace $\HH\subset C^r(\T^d;\R)$  is said to be saturating if  $\HH_\ty $   is dense in $ C^r(\T^d;\R)$. 
  \end{definition}  
We   assume that the following condition is satisfied.
\begin{description}\item[\hypertarget{H1}]{\bf (H$_1$)}  The field    $Q=(Q_1, \ldots, Q_q) $ is saturating, i.e., the subspace $$\HH=\lspan\{Q_j:j=1,\ldots, q\}$$ is saturating in the sense of Definition~\ref{D:2.1}\end{description} 
In this section, we prove the    following result. As we will see below, it implies Theorems~A and~B 
formulated in the Introduction. 
  \begin{theorem}\label{T:2.2}Assume that Condition~{\rm(\hyperlink{H1}{H$_1$})} is satisfied.~Then   for any $\e>0$, $\varkappa>0$,  $\psi_0 \in   H^s$,  and     $\theta \in   C^r(\T^d;\R)$, there is a time $T\in (0, \varkappa)$ and   a control $u\in \Theta(\psi_0,T)$     such~that
	$$ 
 \|\RR_T(\psi_0,u)- e^{i\theta}\psi_0\|_{s}<\e.
 $$ 
\end{theorem}   	 
 \begin{proof}  	 
 By using an induction argument in $N$, we    show that the approximate~controllability property in this theorem 
is true   for any~$\theta \in \HH_N $ and $N\ge0$. Combined with the     saturation hypothesis, this  will   lead 
to approximate   controllability   with any~$\theta \in   C^r(\T^d;\R)$.

 {\it Step~1.~Case $N=0$.}~Let us   show that,
 for any  $\e>0$, $\varkappa>0$,  $\psi_0 \in   H^s$,  and~$\theta \in   \HH$, there is  a time $T\in (0, \varkappa)$ and  a control  $u\in \Theta(\psi_0,T)$     such~that
\begin{equation}\label{2.2}
 \|\RR_T(\psi_0,u)- e^{i\theta}\psi_0\|_{s}<\e.
\end{equation}
By applying Proposition~\ref{P:1.2} with $\varphi=0$ and $u\in \R^q$ such that 
$\theta=-\lag u, Q\rag$, we~obtain
$$
 \RR_{\de}(\psi_0,\de^{-1} u)\to  e^{i   \theta }\psi_0  \quad\text{in $H^s$ as $\de\to 0^+$}.
$$
This implies \eqref{2.2}    with sufficiently small time $T=\delta$ and control~$\delta^{-1}u$.

{\it Step~2.~Case   $N\ge1$.}~We    assume  that the      result  is true   for any $\theta\in \HH_{N-1}$. Let $\tilde \theta\in \HH_N$  be    of the~form
$$
\tilde \theta=\theta_0 -\sum_{j=1}^n \B(\theta_j),	
$$
  where      $n\ge1$ and~$\theta_j \in \HH_{N-1}$, $j=0, \ldots, n$.~By applying Proposition~\ref{P:1.2} 
with $\varphi=\theta_1$ and $u=0$, we get
$$
 	e^{-i\delta^{-1/2}\theta_1} \RR_{\de}(e^{i\delta^{-1/2}\theta_1}\psi_0,0)\to  e^{-i  \B(\theta_1) }\psi_0  \quad\text{in $H^s$ as $\de\to 0^+$}.
$$
The induction hypothesis, the assumption that  $\theta_1\in \HH_{N-1}$, and Proposition~\ref{P:1.1} imply that,  for any  $\e>0$ and $\varkappa>0$,   there is a time  $T_1\in (0, \varkappa)$ and a control~$u_1\in 
\Theta(\psi_0,T_1)$     such~that
	$$ 
 \|\RR_{T_1}(\psi_0,u_1)- e^{-i\B(\theta_1)}\psi_0\|_{s}<\e.
 $$By iterating this argument with   $\theta_j\in \HH_{N-1}$, $j=0, \ldots, n$, we obtain that 
for any~$\e>0$ and~$\varkappa>0$,   there is~$T_n\in (0, \varkappa)$ and~$u_n\in \Theta(\psi_0,T_n)$     such~that
	$$ 
 \|\RR_{T_n}(\psi_0,u_n)- e^{i\left(\theta_0-\sum_{j=1}^n \B(\theta_j)\right)}\psi_0\|_{s}=\|\RR_{T_n}(\psi_0,u_n)- e^{i\tilde \theta }\psi_0\|_{s}<\e.
 $$As $\tilde \theta\in \HH_N$ is arbitrary, this proves the required property for $N$.

  {\it Step~3. Conclusion.} Finally, let   $\theta \in   C^r(\T^d;\R)$ be     arbitrary.  By~the saturation hypothesis,      $  \HH_{\ty}$  is dense in~$C^r(\T^d;\R)$. Hence, we can find  $N\ge1$ and $\tilde \theta \in H_N$ such that 
  $$
  \| e^{i\theta}\psi_0-e^{i\tilde \theta}\psi_0\|_s<\e.
  $$    Applying  the controllability property proved in the previous      steps  for~$\tilde \theta \in H_N$, we~complete the proof. 
  \end{proof}
  As a consequence of this result, we have the following two theorems.
\begin{theorem}\label{T:2.3}
Under the conditions of Theorem~\ref{T:2.2},  for any $M>0$, $\varkappa>0$, and non-zero $\psi_0 \in   H^s$,    there is a time $T\in (0, \varkappa)$ and   a control $u\in \Theta(\psi_0,T)$     such~that
	$$ 
 \|\RR_T(\psi_0,u)\|_s> M.
 $$ 
\end{theorem}
\begin{proof}
It suffices to  	apply Theorem~\ref{T:2.2} by choosing $\theta \in   C^r(\T^d;\R)$ such that 
$$
\|e^{i\theta}\psi_0\|_s> M.
$$To find such $\theta$, we   take any $\theta_1 \in   C^r(\T^d;\R)$ verifying 
$
\|e^{i\theta_1}\psi_0\|_1\neq 0
$, put~$\theta= \la \theta_1$ with sufficiently large $\la>0$, and use the inequality $\|\cdot\|_1\le \|\cdot\|_s.$
  \end{proof}

 \begin{theorem}\label{T:2.4} Assume that the conditions of Theorem~\ref{T:2.2} are satisfied  and
 \begin{equation}\label{2.3}
    \mek \in \lspan\{Q_j:j=1,\ldots, q\} \quad\quad\text{and}\quad\quad V= 0.
    \end{equation} 
 Then, for any $\e>0$,   $l, m\in \Z^d$,~$ \theta\in   C^r(\T^d;\R)$,     and   $T>0$, there is a   control $u\in \Theta(\phi_l,T)$  such that  
	$$ 
 \|\RR_T(\phi_l,u) - e^{i\theta}\phi_m\|_{L^2}<\e.
 $$  
 \end{theorem}
 \begin{proof}Let us take any $  \theta_1\in   C^r(\T^d;\R)$.~Applying Theorem~\ref{T:2.2}, we find a time~$T_1\in (0,T)$ and a  control $u_1\in \Theta(\phi_l,T_1) $ such that
$$
 	 \|\RR_{T_1}(\phi_l,u) - e^{i\theta_1}\phi_l\|_{s}<\frac\e2.	
$$ Choosing  $  \theta_1\in   C^r(\T^d;\R)$ such that 
$$
 	 \|e^{i\theta_1}\phi_l - e^{i\theta}\phi_m\|_{L^2}<\frac\e2,	
$$    we arrive at
$$
 	 \|\RR_{T_1}(\phi_l,u) -  e^{i\theta}\phi_m\|_{L^2}<\e.	
$$
Now, notice that $\phi_l$ is a stationary solution of Eq.~\eqref{0.1} corresponding a   control $u_0\in L^2_{loc}(\R_+;\R^q)$ satisfying the relation
$$
\lag u_0(t), Q(x)\rag= -|l|^2-\kappa (2\pi)^{-dp} \quad \text{for any $t\ge0$ and $x\in \T^d$}.
$$Such a choice of $u_0$ is possible in view of assumption~\eqref{2.3}.  Thus, $u_0\in \Theta(\phi_l, t)$ and 
$\phi_l=\RR_{t}(\phi_l,u_0)$ for any $t\ge0$. Setting $$
u(t)=\begin{cases} u_0(t) & \text{for }t\in[0,T-T_1],   \\ u_1(t-T+T_1) & \text{for }t\in(T-T_1, T],
\end{cases}	
$$ we   complete the proof of the theorem.
 \end{proof}

Let us close this section with an example of a saturating subspace.
 Let~$\II\subset \Z^d_*$  be a    finite     set and let 
 \begin{equation}\label{2.4}
  \HH= \HH(\II)=\text{span}\left\{{\bf 1},\, \sin\lag x,k\rag,\,\cos\lag x,k\rag: k\in \II\right\}.
  \end{equation}Recall that $\II$ is a   	 generator   if~any vector of~$\Z^d$ is a   linear combination of   vectors of~$\II$ with integer coefficients.   The following proposition is proved in Section~\ref{S:4}.
  \begin{proposition}\label{P:2.5}
  	The subspace $  \HH(\II)$ is saturating in the sense of Definition~\ref{D:2.1} if and only if $\II$ is a generator and for any   $l,m\in \II$, there are vectors $\{n_j\}_{j=1}^k \subset  \II$ such that $l\not\perp n_1$, $n_{j}\not\perp n_{j+1}$, $j=1, \ldots, k-1,$ and $n_k\not\perp m$.
  	 \end{proposition}Clearly, the set 
    	      $\KK \subset \Z^d_*$     defined   by \eqref{0.3} satisfies the condition in this 
proposition.~Therefore, the subspace $\HH(\KK)$ is saturating, and     Theorems~A and~B follow 
from~Theorems~\ref{T:2.2} and~\ref{T:2.4}, respectively.

 \section{Proof of Proposition~\ref{P:1.2}}\label{S:3}

  We start by proving the result in the case when
   $s>d/2$ is an   integer, so~$r=s+2$.
  Let~us fix any $ R>0$ and assume that  $\psi_0  \in H^s$, $\varphi \in C^{r}(\T^d;\R)$,    and~$u \in \R^q$  are such~that   
\begin{equation}\label{3.1}
\|\psi_0\|_s +\|\varphi\|_{C^{r}}+\|u\|_{\R^q}\le R.
\end{equation} For any $\delta>0$, we   denote   $\phi(t)=  	e^{-i\delta^{-1/2}\varphi} \RR_t(e^{i\delta^{-1/2}\varphi}\psi_0,\de^{-1} u)$.~According to Proposition~\ref{P:1.1}, $\phi(t)$   exists up      to some maximal time    $\TT^\delta=\TT(e^{i\delta^{-1/2}\varphi}\psi_0,\de^{-1} u)$, and
$$
 \| e^{i\delta^{-1/2}\varphi}\phi(t)\|_s\to +\infty \quad\text{as~$t\to  \TT^{\delta-}$,
if $\TT^\delta<\ty$.}
$$ We need to show that
   \begin{itemize}
   \item[$(a)$]   there is a constant   $\delta_0>0$ such that $\TT^\delta>\delta$ for any $\delta<\de_0$;  
      	   \item[$(b)$]     the following   limit   holds
$$
 	 \phi(\delta)\to  e^{-i \left(\B(\varphi)+\lag u,Q \rag\right)}\psi_0  \quad\text{in $H^s$ as $\de\to 0^+$}.
$$  \end{itemize}To prove these   properties,
  we  introduce the     functions
 \begin{align}
 	w(t)&=   e^{-i \left(\B(\varphi)+\lag u,Q \rag\right)t}\psi_0^\delta,\label{3.2}\\
 	v(t)&= \phi(\delta t)-w(t),\nonumber
 \end{align}where 
 $\psi_0^\delta \in H^{r}$ is such that\,\footnote{In what follows, $C$    denotes positive constants which  may  change
from line to line. These constants     depend on the parameters     $R, V,  Q, \kappa, p,d,s $, but not on $\delta$.  } 
 \begin{align}
     &\|\psi_0^\delta\|_{s}  \le C  \quad\quad\quad  \quad \text{for } \de\le 1,\label{3.3}\\
   & \|\psi_0^\delta\|_r \le C   \delta^{-1/4}  \,\,\,\,\quad \text{for } \de\le 1,\label{3.4}\\
  &  \|	\psi_0-\psi_0^\delta\|_s\to 0  \quad\,\,\, \text{ as $\delta\to 0^+$}.\nonumber
 \end{align}
For example, we can define $\psi_0^\delta$ by using the  heat semigroup:    $\psi_0^\delta= e^{\delta^{ 1/4} 
\Delta}\psi_0$. In~view of~\eqref{3.1}-\eqref{3.4},  we have
\begin{align} 
\|w(t)\|_{s}&\le C, \quad\,\,\quad\quad t\ge0,\label{3.5}\\
\|w(t)\|_{r}&\le C{\delta^{-1/4}}, \quad t\ge0.\label{3.6}	
\end{align}Furthermore,  
  $v(t)$ is 
   well-defined  for $t<\delta^{-1}\TT^\delta$ and satisfies the equation
    \begin{align}
	i \p_t v&=-  \delta\Delta (v+w)+\delta V(v+w)+ \delta \kappa  |v+w|^{2p}(v+w)\nonumber\\
	&\quad -i\delta^{\frac12}\DB( v+w, \varphi) +\B(\varphi)v+\lag u,Q\rag v, \label{3.7}
\end{align} 
     and  the   initial condition 
     \begin{equation}\label{3.8}
      v(0)=\psi_0-\psi_0^\delta,
      \end{equation}
      where 
$$
\DB(v+w, \varphi)= (v+w) \Delta \varphi +2\sum_{j=1}^d  \p_{x_j} (v+w) \,  \p_{x_j}\varphi.
$$Let $\alpha=(\alpha_1, \ldots, \alpha_d)\in \N^d$ be   such   that $|\alpha|=|\alpha_1|+\ldots+|\alpha_d|\le s$.
We take the scalar product of Eq.~\eqref{3.7} with $\p^{2\alpha }v$  in $L^2$ and integrating  by 
parts, we obtain
\begin{align}
\p_t\| \p^{\alpha}v\|_{L^2}^2&\le C\Big(\delta |\lag\Delta w, \p^{2\alpha }v\rag_{L^2}|+\delta |\lag V(v+ w), \p^{2\alpha }v\rag_{L^2}| \nonumber\\ &\quad + \delta  |\lag   |v+w|^{2p}(v+w), \p^{2\alpha }v\rag_{L^2}| +  \delta^{1/2}|\lag \DB( v+w, \varphi), \p^{2\alpha }v\rag_{L^2}| \nonumber\\ &\quad+|\lag \B(\varphi)v+\lag u,Q\rag v, \p^{2\alpha }v\rag_{L^2}|\Big) = \sum_{j=1}^5 I_j.\label{3.9}	
\end{align}We estimate the terms $I_1, I_2, I_3,$ and $I_5$ by integrating by parts and  by 
using \eqref{3.1},~\eqref{3.5}, and \eqref{3.6}:
\begin{align*}
|I_1|&\le C \de \|w\|_{r} \|v\|_s  \le  C \delta^{3/4} \|v\|_s, \\
 |I_2|&\le C\delta \|v+w\|_s \|v\|_s \le C \delta  \|v\|_s^2+C \delta  \|v\|_s,\\
 |I_3|&\le C \delta \|v+w\|_s^{2p+1} \|v\|_s \le  C\delta \|v\|_s^{2(p+1)} + C \delta  \|v\|_s,\\	
|I_5|&\le 	C   \|v\|_s^2.
\end{align*}
  We estimate   $I_4$ as follows
 $$
|I_4| \le C \delta^{1/2}  \|v\|_s^2+C \delta^{1/2}  \|w\|_{s+1}\|v\|_s\le C \delta^{1/2}  \|v\|_s^2+C \delta^{1/4}  \|v\|_s,
$$ In the last relation, we used again the  integration by parts, the 
identities~\eqref{3.1},~\eqref{3.5} and \eqref{3.6},  and the equality
\begin{align*}
\lag \p_{x_j} \varphi\, \p_{x_j} \p^\alpha v, \p^\alpha v \rag_{L^2} &=\frac12\lag \p_{x_j} \varphi   , \p_{x_j}|\p^\alpha v|^2 \rag_{L^2} =-\lag \p^2_{x_j} \varphi   ,  |\p^\alpha v|^2 \rag_{L^2}.
\end{align*}
Summing up inequalities \eqref{3.9} for all $\alpha\in \N^d$, $|\alpha|\le s$,
  combining the resulting
inequality with the estimates for $I_j$ and the Young inequality, and recalling
that~$\de \le  1$, we obtain
$$
 \p_t \| v \|^2_s    \le C\delta^{1/2}+  C(1+\delta^{1/2})\|v\|_s^2+ C\delta \|v\|_s^{2(p+1)} ,  \quad   t\le  \delta^{-1}\TT^\delta.
$$This inequality, together with \eqref{3.8} and  the Gronwall inequality, implies that
\begin{equation}\label{3.10}
\|v(t)\|_s^2 \le e^{C(1+\delta^{1/2})t} \left(C \delta^{1/2} t+ \|\psi_0-\psi_0^\delta\|_s^2+C\delta\int_0^t \|v(y)\|_s^{2(p+1)}\dd y\right)
\end{equation}
 for $   t\le  \delta^{-1}\TT^\delta$. 
Let us take $\delta_0\in (0,1)$ so small that, for $\delta<\delta_0$, 
\begin{gather} 
\|\psi_0-\psi_0^\delta\|_s^2< 1,\label{3.11}\\
e^{C(1+\delta^{1/2})}\left(C\delta^{1/2}+\|\psi_0-\psi_0^\delta\|_s^2\right)<\frac12,  \label{3.12}
\end{gather} 
 and denote
$$
\tau^\de=\sup\left\{t<  \delta^{-1}\TT^\delta: \|v(t)\|_s<1\right\}.
$$From \eqref{3.8} and \eqref{3.11}  it follows that   $ \tau^\de>0$ for $\delta<\delta_0$. Let us show that $\tau^\de>1$ provided that
\begin{equation}\label{3.13}
	\delta_0< \left(2Ce^{2C}\right)^{-1}.
\end{equation}
Assume, by contradiction, that $\tau^\de \le 1$.
 Let $t=\tau^\de$ in~\eqref{3.10}. By using~\eqref{3.12} and~\eqref{3.13}, we obtain
$$
1=\|v(\tau^\de)\|_s^2< \frac12+\frac12 \int_0^{\tau^\de}\|v(y)\|_s^{2(p+1)}\dd y\le 1.
$$ This contradiction shows that $   \tau^\de> 1$ for $\delta<\delta_0$, hence also $1< \delta^{-1}\TT^\delta$. 
Thus, property (a) is proved.   Taking $t=1$ in~\eqref{3.10}, we arrive at
$$
\|v(1)\|_s^2\le e^{C(1+\delta^{1/2})} \left(C\delta^{1/2}+ \|\psi_0-\psi_0^\delta\|_s^2+C\delta\right)\to 0 \quad \text{as $\delta\to 0^+$}.
$$This implies (b) and completes the proof in the case when $s>d/2$ is an   integer.

To derive properties (a) and (b)  in the general case, i.e., when
  $s\ge s_d$ is an arbitrary number, we use inequality~\eqref{3.10} for integer values of~$s$   and an interpolation argument.

 \section{Saturating subspaces}
\label{S:4}

 \begin{proof}[Proof of Proposition~\ref{P:2.5}] The proof is  divided into four steps.
 
  {\it Step~1.}  First, let us assume that  $\II\subset \Z^d_*$  is an arbitrary finite set,~$\HH_0(\II)=\HH(\II)$ is the subspace defied by~\eqref{2.4},     $\HH_j(\II)=\FF(\HH_{j-1}(\II))$ for~$j\ge1$, and $\HH_\ty(\II)$ is defined by \eqref{2.1}. 
  
  {\it Step~1.1.}  Let us show that, if 
$$
   	\cos\lag x,   m\rag, \, 	\sin\lag x,   m\rag    \in \HH_\ty(\II) \quad \text{for some $m\in \Z^d_*$},
$$then  
$$
   	\B(\cos\lag x,  m\rag),\,    	\B(\sin\lag x,  m\rag)     \in \HH_\ty(\II).
$$Indeed, assume that 
  \begin{equation}\label{4.1}
   \cos\lag x,   m\rag, \, \sin\lag x,   m\rag    \in   \HH_N(\II) \quad \text{ for some $N\ge0$}.
   \end{equation}
      The equalities
\begin{equation}\label{4.2}
    	\cos\lag x,2 m\rag =1-\frac 2{|m|^2} \B(\cos\lag x,  m\rag)=\frac 2{|m|^2} \B(\sin\lag x,  m\rag)-1,
\end{equation}   the assumptions       ${\bf 1} \in \HH(\II)$ and \eqref{4.1}, and the   definition of $\FF$ imply  that 
\begin{equation}\label{4.3}
 \cos\lag x,2 m\rag \in \HH_{N+1}(\II). 
\end{equation}As a  consequence of \eqref{4.2} and \eqref{4.3}, we   have
\begin{align*} 
   	\B(\cos\lag x,  m\rag)&= \frac{|m|^2}{2}(1-\cos\lag x,2 m\rag) \in \HH_{N+1}(\II),\\
    \B(\sin\lag x,  m\rag) &= \frac{|m|^2}{2}(1+\cos\lag x,2 m\rag)   \in \HH_{N+1}(\II),
   \end{align*}which imply the required result.
    
    {\it Step~1.2.} Let us  show that, if 
$$
   	\cos\lag x,   m\rag, \, 	\sin\lag x,   m\rag,\, 	\cos\lag x,   l\rag, \, 	\sin\lag x,   l\rag    \in \HH_\ty(\II) 
$$for some $m, l\in \Z^d_*$ such that 
 $m \not\perp l$, then
$$
   	\cos\lag x,   m+l\rag, \, 	\sin\lag x,   m+l\rag  \in \HH_\ty(\II). 
$$
   Indeed,  this follows immediately from the equalities
   \begin{align*}
   		\cos\lag x,   m+l\rag = & \, \pm\frac{1}{\lag m, l\rag}\Big(\B(\sin\lag x,   m\rag\pm\sin\lag x,   l\rag)+\B(\cos\lag x,   m\rag\mp\cos\lag x,   l\rag)\\&  - \B(\sin\lag x,   m\rag)-\B(\sin\lag x,   l\rag)-\B(\cos\lag x,   m\rag)-\B(\cos\lag x,   l\rag)\Big),\\
   		\sin\lag x,   m+l\rag = & \, \pm\frac{1}{\lag m, l\rag}\Big(\B(\sin\lag x,   m\rag\mp\cos\lag x,   l\rag)+\B(\cos\lag x,   m\rag\mp\sin\lag x,   l\rag)\\&  - \B(\sin\lag x,   m\rag)-\B(\sin\lag x,   l\rag)-\B(\cos\lag x,   m\rag)-\B(\cos\lag x,   l\rag)\Big)
   \end{align*}and the result of step~1.1.

   {\it Step~2.} Now, let us suppose that $\II\subset \Z^d_*$ is a finite set    such that, for any~$l,m\in \II$, there are vectors $\{n_j\}_{j=1}^k \subset  \II$ satisfying $l\not\perp n_1$, $n_{j}\not\perp n_{j+1}$, $j=1, \ldots, k-1,$ and $n_k\not\perp m$.~Let $N=\text{card}(\II)$ and $\II=\{m_1, \ldots, m_N\}$. Arguing by induction on  $N$,  we show in this step   
that \begin{equation}\label{4.4}
       	\cos\lag x, a_1 m_1+\ldots+ a_N m_N  \rag, \, 	\sin\lag x, a_1 m_1+\ldots+ a_N m_N   \rag  \in \HH_\ty(\II) 
       \end{equation}for any $a_1, \ldots, a_N\in \Z$.
   
   {\it Step~2.1.}
       Let   $ \II=\{m_1, m_2\}\subset \Z^d_*$ with $m_1\not \perp m_2$.       By the result of step~1.2, we have 
$$
       	\cos\lag x, a_1 m_1  \rag, \, 	\sin\lag x, a_1 m_1  \rag ,\, 	\cos\lag x,   a_2 m_2 \rag, \, 	\sin\lag x,   a_2 m_2   \rag \in \HH_\ty(\II)  
$$for any $a_1,a_2\in \Z$. Again, in view of step~1.2,  this   implies that 
$$
       	\cos\lag x, a_1 m_1+ a_2 m_2  \rag, \, 	\sin\lag x, a_1 m_1+ a_2 m_2   \rag  \in \HH_\ty(\II) 
    $$for any $a_1,a_2\in \Z$. 
         
         {\it Step~2.2.}~Assume that the required property is true  if the cardinality of the set~$\II$ is less or equal to $N-1$.~Let $\II\subset \Z^d_*$ be such that $N=\text{card}(\II)$ and~$\II=\{m_1, \ldots, m_N\}$. 
         Without loss of generality, we can assume $m_{N-1}\not \perp m_N$ and the set~$  \{m_1, \ldots, m_{N-1}\}$ satisfies the condition formulated in the beginning of   step~2.~Let us take any~$a_1, \ldots, a_N\in \Z$ and $k\ge1$ and write 
         \begin{align}
         a_1 m_1+\ldots+ a_N m_N& =\left(a_1 m_1+\ldots+ a_{N-2} m_{N-2}+ (a_{N-1}-k) m_{N-1}\right)\nonumber\\& \quad + \left(k  m_{N-1}+a_N m_N \right).\label{4.5}
         \end{align}Then,
        \begin{align*}
         \lag a_1 m_1+\ldots + (a_{N-1}-k) m_{N-1}, k  m_{N-1}+a_N m_N  \rag &=(a_{N-1}-k)k \|m_{N-1}\|^2\\ &\quad + O(k) \quad \text{as $k\to +\ty$.}
       \end{align*}As $m_{N-1}\neq 0$, for sufficiently large $k\ge1$, we have 
    \begin{align}
       a_1 m_1+\ldots+ a_{N-2} m_{N-2}+ (a_{N-1}-k) m_{N-1}\not\perp k  m_{N-1}+a_N m_N.\label{4.6}
\end{align}
Relation  \eqref{4.4} is proved by combining  \eqref{4.5} and \eqref{4.6}, the 
induction hypothesis, and the assumption that
$m_{N-1}\not \perp m_N$.
         
         {\it Step~3.}~We conclude from step~2 that,  if $\II\subset \Z^d_*$ is a   set satisfying  the conditions of Proposition~\ref{P:2.5}, then
$$
\cos\lag x,m\rag,\, \sin\lag x,m\rag \in \HH_\ty(\II)  \quad \text{for any $m\in \Z^d_*$.}
$$ This implies that $\HH_\ty(\II)$ is dense in   $C^r(\T^d;\R) $ for any $r\ge0$, hence $\HH(\II)$ is saturating.

  {\it Step~4.} Finally, let us assume  that the conditions of the proposition are not satisfied for~$\II \subset \Z^d_*$.  We   distinguish between two cases.
  
   {\it Step~4.1.}
If $\II$ is not a generator,    we can find a vector~$n\in \Z^d_*$   which does not belong   to the set $\tilde \II$ of  linear combinations of vectors of~$\II$ with integer coefficients. It~is easy to see  that 
   $$
   \HH_\ty(\II)\subset\text{span} \{  \sin \lag  x,  m\rag,\,  \cos\lag x,m \rag:\,\,   m\in  \tilde \II\}.
   $$
   Thus,  the functions   $\sin\lag x,n\rag$  and   $\cos\lag x,n\rag$   are orthogonal to the vector 
space~$\HH_\ty(\II)$ in  the Sobolev spaces $H^j(\T^d;\R)$ for any~$j\ge0$.~We conclude that     $\HH_\ty(\II)$ is 
not dense in~$C^r(\T^d;\R) $, thus   the subspace~$\HH (\II)$ is not saturating.   
   
      {\it Step~4.2.}  If $\II$ does not satisfy the second condition in the theorem, then it is of the form
      $$
      \II=\cup_{j=1}^k\{m_1^j, \ldots, m^j_{n_j}\},
      $$where $k\ge2$ and $m_{i_1}^{j_1} \bot m_{i_2}^{j_2}$ for any integers $1\le j_1<j_2\le k,$ $1\le i_1 \le 
n_{j_1},$ and~$1\le i_2 \le n_{j_2}$.~By using the arguments of the steps~1 and~2, it is easy to verify that    
the function   $\cos\lag x, m_1^{j_1}+m_2^{j_2}  \rag$   is orthogonal to~$\HH_\ty(\II)$ in~$H^j(\T^d;\R)$ for 
any~$j\ge0$. Thus, the space $\HH_\ty(\II)$ is not dense in~$C^r(\T^d;\R) $.
         \end{proof}
 
   \section{Growth of Sobolev norms}\label{S:5}

Let us  consider the   NLS equation
    \begin{align}
 	i \p_t \psi&=-\Delta \psi +V(x)\psi  +\kappa  |\psi|^{2p} \psi +\lag \eta(t),Q(x) \rag \psi,  \label{5.1}\\ 
 	\psi(0)&=\psi_0 \label{5.2}
 \end{align}  with  potential $V$   and parameters $d, p,\kappa$ as in the  previous sections. We   assume that   the   field $Q$ satisfies   Condition~{\rm(\hyperlink{H1}{H$_1$})}  and~$\eta$ is a random process  of the form~\eqref{0.6}
 with the following   condition satisfied for the  random variables $\{\eta_k\}$. We~denote~$J=[0,1]$ and $\EE=L^2(J; \R^q)$.
\begin{description}
   \item[\hypertarget{H2}]{\bf (H$_2$)} $\{\eta_k\}$ are independent   random variables in $\EE$ with common  law~$\ell$ such~that    
     $$ 
    \int_E \|y\|_\EE^2 \,\ell(\dd y)<\ty\quad\quad \text{and}  \quad \quad \supp\ell=\EE.
     $$ 
\end{description}
For example, this condition is satisfied   if  the random variables $\{\eta_k\}$ are of the~form
$$
\eta_k(t)=\sum_{j=1}^{+\infty} b_j\xi_{jk}e_j(t),  \quad  t\in J,
$$where 
$\{b_j\}$ are non-zero real numbers verifying  
$\sum_{j=1}^{+\infty} b_j^2<\infty,$ $\{e_j\}$ is an orthonormal basis in $\EE$, and $\{\xi_{jk}\}$ are independent real-valued random variables   whose law has a continuous density $\rho_j$ with respect to the
Lebesgue measure such that
$$
\int_{-\ty}^{+\ty} x^2  \rho_j(x) \,\dd x=1, \quad \rho_j(x)>0 \quad\text{for all $x\in\R$ and $j\ge1$.}
$$

By~Proposition~\ref{P:1.1}, the problem~\eqref{5.1},~\eqref{5.2} is locally well-posed in $H^s$ for any~$s>d/2$ up 
to some  (random)  maximal time $\TT=\TT(\psi_0, \eta)>0$.
  Let~$\pP_{\psi_0}$ be the probability measure corresponding to the trajectories issued from~$\psi_0$ (e.g.,~see~Section~1.3.1 in~\cite{KS-book}).
  Recall that $\sS$ is the unit sphere in $L^2$.

   \begin{theorem}\label{T:5.1}
Under the  Conditions~{\rm(\hyperlink{H1}{H$_1$})} and~{\rm(\hyperlink{H2}{H$_2$})}, 
    for any $s>s_d$ and any~$\psi_0\in H^s \cap \sS$, we~have
\begin{equation}\label{5.3}
\pP_{\psi_0}\left\{\limsup_{t\to \TT^-}\|\psi(t)\|_s=+\ty\right\}=1.
\end{equation}   
\end{theorem} 
By the blow-up alternative,    equality \eqref{5.3} gives new information in the case~$\TT(\psi_0, \eta)=+\ty.$
\begin{proof}{\it Step~1.~Reduction.} 
Together with Eq.~\eqref{5.1}, let us   consider    the following truncated NLS equation:
 \begin{equation}\label{5.4}
 	i \p_t \psi=-\Delta \psi +V(x)\psi  +\kappa  \chi_R(\|\psi\|_s)  |\psi|^{2p} \psi +\lag \eta(t),Q(x) \rag \psi,   
 \end{equation}
where $R>0$ and $\chi_R \in C^\ty_0(\R)$ is  such that  $0\le \chi_R(x)\le 1 $ for~$x\in \R$ and $\chi_R(x)=1$ 
for~$|x|\le R$.~Let $\FF_k$, $k\ge1$ be the $\sigma$-algebra generated by the family~$\{\eta_j\}_{j=1}^{k}$. The 
problem~\eqref{5.4},~\eqref{5.2} is   globally well-posed. The following proposition is proved at the end of this 
section.
 \begin{proposition}\label{P:5.2}
 For any    $\psi_0\in  H^s$ and $R>0$, the~problem~\eqref{5.4},~\eqref{5.2}  has a  unique   solution $\psi^R \in C(\R_+;H^s)$. Moreover, the family   
 $$
 \left\{\psi^R(k+\cdot):J\to H^s\right\}_{k\ge 0}
 $$ defines   an $C(J;H^s)$-valued Markov process with respect to the filtration $\FF_{k+1}$.
 \end{proposition}
Let us fix any   $0<M<R$ and
consider the stopping time
$$
\tau_{M,R}=1+\min\left\{k\ge0: \|\psi^R(k+\cdot)\|_{C(J;H^s)}> M\right\},\quad \psi_0\in H^s,
$$where   the minimum over an empty set is equal to $+\ty$.~Assume   we have shown~that
\begin{equation}\label{5.5}
\pP_{\psi_0} \{\tau_{M,R} <\infty\}=1, \quad \psi_0\in H^s\cap \sS.
\end{equation} Since $R>M$, this   implies that  
\begin{equation}\label{5.6}
\pP_{\psi_0} \{\tau_{M} <\infty\}=1, \quad \psi_0\in H^s\cap \sS,
\end{equation} where 
$$
\tau_{M} =\min\left\{k\ge0: \sup_{t\in J, \, k+t<\TT }\|\psi(k+t)\|_s> M\right\}
$$and again the minimum over an empty set is   $+\ty$.~As  $M>0$  is arbitrary, we~conclude that \eqref{5.3} holds.

 {\it Step~2.~Proof of \eqref{5.5}.}~Assume that   there is an integer $l\ge 1$ such that
 \begin{equation}
 c=c(M,R)=\sup_{{\psi_0\in H^s\cap \sS}}\pP_{\psi_0} \left\{\tau_{M,R} 
>l\right\}<1.\label{5.7}
\end{equation}
 Combining this with  the Markov property, we obtain
\begin{align*}
\pP_{\psi_0}\left\{\tau_{M,R}  >n l\right\}&=\E_{\psi_0} \left(\I_{\{\tau_{M,R} >(n-1)l\}}\pP_{\phi}\{\tau_{M,R}   
>l\}|_{\phi= {\psi^R((n-1)l)}}\right)\nonumber\\&\le  c\, \pP_{\psi_0} \left\{\tau_{M,R} 
>(n-1)l\right\},
\end{align*} where $\E_{\psi_0}$ is the expectation corresponding to $\pP_{\psi_0}$.~Iterating this inequality, we~get
$$
\pP_{\psi_0} \left\{\tau_{M,R} >n l\right\}\le c^n. 
$$ This, together with
  the Borel--Cantelli lemma, implies   \eqref{5.5}. 

 {\it Step~3.~Proof of \eqref{5.7}.} By~Theorem~\ref{T:2.3}, for any~$\psi_0 \in   H^{s_d}\cap \sS$,    there is a       control $u\in \EE $     such~that
\begin{equation}\label{5.8}
\sup_{t\in J, \, t<\TT }\|\psi(t)\|_{s_d}> M.
\end{equation} On the other hand, Condition~{\rm(\hyperlink{H2}{H$_2$})} implies that
$$
\pP\left\{\|u-\eta\|_{\EE}<\de\right\}>0
$$for any $\delta>0$. 
Combining this with Proposition~\ref{P:1.1} and inequality~\eqref{5.8}, we see that there is a number $\delta>0$ such that 
$$
\inf_{\psi_0'\in B_{H^{s_d}}(\psi_0,\delta)\cap \sS}\pP_{\psi_0'}\left\{\sup_{t\in J, \, t<\TT' }\|\psi(t)\|_{s_d}> M\right\}>0,
$$    where $\TT'=\TT(\psi_0',\eta)$.
As $R>M$,  we also have
$$
\inf_{\psi_0'\in B_{H^{s_d}}(\psi_0,\delta)\cap \sS}\pP_{\psi_0'}\left\{\sup_{t\in J}\|\psi^R(t)\|_{s_d}> M\right\}>0.
$$ Since the ball $B_{H^{s}}(0,M)$ is compact in $H^{s_d}$ and $\|\cdot\|_{s_d}\le \|\cdot\|_s$, we derive that 
 $$
\inf_{\psi_0\in B_{H^{s}}(0,M)\cap \sS}\pP_{\psi_0}\left\{\sup_{t\in J}\|\psi^R(t)\|_s> M\right\}>0.
$$ 
The latter and the fact that 
$$
\pP_{\psi_0}\left\{\tau_{M,R}=1\right\}=1  \quad \text{if $\|\psi_0\|_s>M$}
$$
 imply   \eqref{5.7} with $l=1$~and 
$$
c=1-\inf_{\psi_0\in B_{H^{s}}(0,M)\cap \sS}\pP_{\psi_0}\left\{\sup_{t\in J}\|\psi^R(t)\|_s> M\right\}.
$$This completes the proof of    the theorem.
  \end{proof}
\begin{proof}[Proof of Proposition~\ref{P:5.2}]
 The  local well-posedness of \eqref{5.4}, \eqref{5.2} is proved by standard arguments.~As the $H^s$-norm of the 
solution remains bounded on any bounded interval, it can be extended to any $t>0$. For any $k\ge1$, let us denote by~$\psi_k(\psi_0, 
\eta_1, \ldots, \eta_k)$   the restriction of the solution of~\eqref{5.4},~\eqref{5.2} to the interval~$[k-1, k]$ (we skip the dependence on~$R$).  Then $\{\psi_k(\psi_0, \eta_1, \ldots, \eta_k)\}_{k\ge1}$ is a Markov 
process in~$C(J, H^s)$. Indeed,   we~have 
 $$
 \psi_{k+n}(\psi_0, \eta_1, \ldots, \eta_{k+n})= \psi_{n}(\psi_{k}(\psi_0, \eta_1, \ldots, \eta_k), \eta_{k+1}, \ldots, \eta_{k+n}).
 $$As $\{\eta_j\}_{j\ge k+1}$ is independent of $\FF_k$ and $\psi_k$ is $\FF_k$-measurable, the following equality holds
\begin{equation}
 \E   \left(f(\psi_{k+n}(\psi_0, \eta_1, \ldots, \eta_{k+n}))| \FF_k\right) =  \E f(\psi_{n}(\psi, \eta_{k+1}, \ldots, \eta_{k+n}))  \label{5.9}
\end{equation}
 for any bounded measurable function $f:C(J, H^s)\to \R$. Here, $\psi$ is the value  at time-1 of $ 
\psi_{k}(\psi_0, \eta_1, \ldots, \eta_k)$. The vectors $ (\eta_1, \ldots, \eta_n)$ and~$ (\eta_{k+1}, \ldots, 
\eta_{k+n})$ have the same law, so 
 $$
\E f(\psi_{n}(\psi, \eta_{k+1}, \ldots, \eta_{k+n})) =\E f(\psi_{n}(\psi, \eta_1, \ldots, \eta_n)). 
 $$Combining this and \eqref{5.9}, we arrive at the required result.
 \end{proof}

 \addcontentsline{toc}{section}{Bibliography}
 \bibliographystyle{alpha}

\begin{thebibliography}{BGMR18}

\bibitem[AS05]{AS-2005}
A.~A. Agrachev and A.~V. Sarychev.
\newblock Navier--{S}tokes equations: controllability by means of low modes
  forcing.
\newblock {\em J. Math. Fluid Mech.}, 7(1):108--152, 2005.

\bibitem[AS06]{AS-2006}
A.~A. Agrachev and A.~V. Sarychev.
\newblock Controllability of 2{D} {E}uler and {N}avier--{S}tokes equations by
  degenerate forcing.
\newblock {\em Comm. Math. Phys.}, 265(3):673--697, 2006.

\bibitem[AS08]{AS-2008}
A.~A. Agrachev and A.~V. Sarychev.
\newblock Solid controllability in fluid dynamics.
\newblock In {\em Instability in {M}odels {C}onnected with {F}luid {F}lows.
  {I}}, volume~6 of {\em Int. Math. Ser. (N. Y.)}, pages 1--35. Springer, New
  York, 2008.

\bibitem[BC06]{BC-2006}
K.~Beauchard and J.-M. Coron.
\newblock Controllability of a quantum particle in a moving potential well.
\newblock {\em J. Funct. Anal.}, 232(2):328--389, 2006.

\bibitem[BCCS12]{BCCS-12}
U.~Boscain, M.~Caponigro, T.~Chambrion, and M.~Sigalotti.
\newblock A weak spectral condition for the controllability of the bilinear
  {S}chr\"{o}dinger equation with application to the control of a rotating
  planar molecule.
\newblock {\em Comm. Math. Phys.}, 311(2):423--455, 2012.

\bibitem[BCT18]{BCT2}
K.~Beauchard, J.-M. Coron, and H.~Teismann.
\newblock Minimal time for the approximate bilinear control of
  {S}chr\"{o}dinger equations.
\newblock {\em Math. Methods Appl. Sci.}, 41(5):1831--1844, 2018.

\bibitem[Bea05]{B-2005}
K.~Beauchard.
\newblock Local controllability of a 1-{D} {S}chr\"{o}dinger equation.
\newblock {\em J. Math. Pures Appl. (9)}, 84(7):851--956, 2005.

\bibitem[BGMR18]{MR3738262}
D.~Bambusi, B.~Gr\'{e}bert, A.~Maspero, and D.~Robert.
\newblock Reducibility of the quantum harmonic oscillator in {$d$}-dimensions
  with polynomial time-dependent perturbation.
\newblock {\em Anal. PDE}, 11(3):775--799, 2018.

\bibitem[BL10]{BL-2010}
K.~Beauchard and C.~Laurent.
\newblock Local controllability of 1{D} linear and nonlinear {S}chr\"{o}dinger
  equations with bilinear control.
\newblock {\em J. Math. Pures Appl. (9)}, 94(5):520--554, 2010.

\bibitem[Bou99]{Bourgain-1999}
J.~Bourgain.
\newblock On growth of {S}obolev norms in linear {S}chr\"{o}dinger equations
  with smooth time dependent potential.
\newblock {\em J. Anal. Math.}, 77:315--348, 1999.

\bibitem[Bou00]{Bourgain-2000}
J.~Bourgain.
\newblock Problems in {H}amiltonian {PDE}'s.
\newblock Number Special Volume, Part I, pages 32--56. 2000.

\bibitem[Caz03]{C-2003}
T.~Cazenave.
\newblock {\em Semilinear {S}chr\"{o}dinger equations}, volume~10.
\newblock AMS, Providence, RI, 2003.

\bibitem[CKS{\etalchar{+}}10]{CKSTT-2010}
J.~Colliander, M.~Keel, G.~Staffilani, H.~Takaoka, and T.~Tao.
\newblock Transfer of energy to high frequencies in the cubic defocusing
  nonlinear {S}chr\"{o}dinger equation.
\newblock {\em Invent. Math.}, 181(1):39--113, 2010.

\bibitem[CMSB09]{CMSB-2009}
T.~Chambrion, P.~Mason, M.~Sigalotti, and U.~Boscain.
\newblock Controllability of the discrete-spectrum {S}chr\"{o}dinger equation
  driven by an external field.
\newblock {\em Ann. Inst. H. Poincar\'{e} Anal. Non Lin\'{e}aire},
  26(1):329--349, 2009.

\bibitem[Del14]{Delort-14}
J.-M. Delort.
\newblock Growth of {S}obolev norms for solutions of time dependent
  {S}chr\"{o}dinger operators with harmonic oscillator potential.
\newblock {\em Comm. Partial Differential Equations}, 39(1):1--33, 2014.

\bibitem[EgKS03]{EBK-2003}
M.~B. Erdo\~{g}an, R.~Killip, and W.~Schlag.
\newblock Energy growth in {S}chr\"{o}dinger's equation with {M}arkovian
  forcing.
\newblock {\em Comm. Math. Phys.}, 240(1-2):1--29, 2003.

\bibitem[GG17]{GG-2017}
P.~G\'{e}rard and S.~Grellier.
\newblock The cubic {S}zeg\"{o} equation and {H}ankel operators.
\newblock {\em Ast\'{e}risque}, (389):vi+112, 2017.

\bibitem[GK15]{GK-2015}
M.~Guardia and V.~Kaloshin.
\newblock Growth of {S}obolev norms in the cubic defocusing nonlinear
  {S}chr\"{o}dinger equation.
\newblock {\em J. Eur. Math. Soc.}, 17(1):71--149, 2015.

\bibitem[HM20]{HM-20}
E.~Haus and A.~Maspero.
\newblock Growth of {S}obolev norms in time dependent semiclassical anharmonic
  oscillators.
\newblock {\em J. Funct. Anal.}, 278(2):108316, 25, 2020.

\bibitem[HPTV15]{HPT-2015}
Z.~Hani, B.~Pausader, N.~Tzvetkov, and N.~Visciglia.
\newblock Modified scattering for the cubic {S}chr\"{o}dinger equation on
  product spaces and applications.
\newblock {\em Forum Math. Pi}, 3, 2015.

\bibitem[KS12]{KS-book}
S.~Kuksin and A.~Shirikyan.
\newblock {\em Mathematics of {T}wo-{D}imensional {T}urbulence}.
\newblock Cambridge University Press, Cambridge, 2012.

\bibitem[Kuk97]{SBK-97}
S.~Kuksin.
\newblock Oscillations in space-periodic nonlinear {S}chr\"{o}dinger equations.
\newblock {\em Geom. Funct. Anal.}, 7(2):338--363, 1997.

\bibitem[Mas19]{Maspero-19}
A.~Maspero.
\newblock Lower bounds on the growth of {S}obolev norms in some linear time
  dependent {S}chr\"{o}dinger equations.
\newblock {\em Math. Res. Lett.}, 26(4):1197--1215, 2019.

\bibitem[Mir09]{M-09}
M.~Mirrahimi.
\newblock Lyapunov control of a quantum particle in a decaying potential.
\newblock {\em Ann. Inst. H. Poincar\'{e} Anal. Non Lin\'{e}aire},
  26(5):1743--1765, 2009.

\bibitem[Ner09]{VN-2009}
V.~Nersesyan.
\newblock Growth of {S}obolev norms and controllability of the
  {S}chr\"{o}dinger equation.
\newblock {\em Comm. Math. Phys.}, 290(1):371--387, 2009.

\bibitem[Ner10]{VN-2010}
V.~Nersesyan.
\newblock Global approximate controllability for {S}chr\"{o}dinger equation in
  higher {S}obolev norms and applications.
\newblock {\em Ann. Inst. H. Poincar\'{e} Anal. Non Lin\'{e}aire}, 27(3), 2010.

\bibitem[Ner21]{VN-2019A}
V.~Nersesyan.
\newblock {Approximate controllability of nonlinear parabolic PDEs in arbitrary
  space dimension}.
\newblock {\em To appear in Math. Control Relat. Fields}, 2021.

\bibitem[Sar12]{Sar-2012}
A.~Sarychev.
\newblock Controllability of the cubic {S}chr\"odinger equation via a
  low-dimensional source term.
\newblock {\em Math. Control Relat. Fields}, 2(3):247--270, 2012.

\bibitem[Shi06]{shirikyan-cmp2006}
A.~Shirikyan.
\newblock Approximate controllability of three-dimensional {N}avier--{S}tokes
  equations.
\newblock {\em Comm. Math. Phys.}, 266(1):123--151, 2006.

\bibitem[Shi07]{shirikyan-aihp2007}
A.~Shirikyan.
\newblock Exact controllability in projections for three-dimen\-sional
  {N}avier--{S}tokes equations.
\newblock {\em Ann. Inst. H. Poincar\'e Anal. Non Lin\'eaire}, 24(4):521--537,
  2007.

\bibitem[Shi18]{Shir-2018}
A.~Shirikyan.
\newblock Control theory for the {B}urgers equation: {A}grachev-{S}arychev
  approach.
\newblock {\em Pure Appl. Funct. Anal.}, 3(1):219--240, 2018.

\bibitem[Tao06]{TT-2006}
T.~Tao.
\newblock {\em Nonlinear dispersive equations}, volume 106.
\newblock AMS, Providence, RI, 2006.
\newblock Local and global analysis.

\end{thebibliography}

\newcommand{\etalchar}[1]{$^{#1}$}
\def\cprime{$'$} \def\cprime{$'$}
  \def\polhk#1{\setbox0=\hbox{#1}{\ooalign{\hidewidth
  \lower1.5ex\hbox{`}\hidewidth\crcr\unhbox0}}}
  \def\polhk#1{\setbox0=\hbox{#1}{\ooalign{\hidewidth
  \lower1.5ex\hbox{`}\hidewidth\crcr\unhbox0}}}
  \def\polhk#1{\setbox0=\hbox{#1}{\ooalign{\hidewidth
  \lower1.5ex\hbox{`}\hidewidth\crcr\unhbox0}}} \def\cprime{$'$}
  \def\polhk#1{\setbox0=\hbox{#1}{\ooalign{\hidewidth
  \lower1.5ex\hbox{`}\hidewidth\crcr\unhbox0}}} \def\cprime{$'$}
  \def\cprime{$'$} \def\cprime{$'$} \def\cprime{$'$}

\end{document}